\documentclass[12pt]{article}
\usepackage{amssymb,latexsym,amsmath,enumerate,verbatim,amsfonts,amsthm}
\usepackage[active]{srcltx}

\numberwithin{equation}{section}
\newtheorem{Definition}{Definition}[section]
\newtheorem{thm}{Theorem}[section]
\newtheorem{cor}[thm]{Corollary}
\newtheorem{lem}[thm]{Lemma}

\begin{document}
\title{Pazy's fixed point theorem with respect to the partial order  in  uniformly convex Banach spaces}

\author{Yisheng Song\thanks{Corresponding author. School of Mathematics and Information Science and Henan Engineering Laboratory for Big Data Statistical Analysis and Optimal Control, Henan Normal University, XinXiang HeNan, P.R. China, 453007.
Email: songyisheng@htu.cn. The work was supported by the National Natural Science Foundation of P.R. China (Grant No. 11571095),  Program for Innovative Research Team (in Science and Technology)  in University of Henan Province (14IRTSTHN023).}\quad Rudong Chen\thanks{Department of Mathematics, Tianjin Polytechnic University, Tianjin, P.R. China. Email: chenrd@tjpu.edu.cn. The work was supported by the National Natural Science
Foundation of P.R. China (Grant No.11071279) and  was supported by
The Science and Technology Plans of Tianjin (No. 15PTSYJC00230) }}

\date{}

 \maketitle

\begin{abstract}
	In this paper,  the Pazy's Fixed Point Theorems of monotone $\alpha-$nonexpansive mapping $T$ are proved in a uniformly convex Banach space $E$ with the partial order ``$\leq$". That is, we obtain that the fixed point set of $T$ with respect to the partial order ``$\leq$" is nonempty whenever the Picard iteration $\{T^nx_0\}$  is bounded for some  initial point $x_0$ with  $x_0\leq Tx_0$ or  $Tx_0\leq x_0$. When restricting the demain of $T$ to the cone $P$,  a monotone $\alpha-$nonexpansive mapping $T$ has at least a fixed point if and only if the Picard iteration $\{T^n0\}$ is bounbed.   Furthermore, with the help of the properties of the normal cone $P$, the weakly and strongly convergent theorems of the Picard iteration $\{T^nx_0\}$ are showed for finding a fixed point  of $T$ with respect to the partial order ``$\leq$"  in uniformly convex ordered Banach space. \\
	
	
\end{abstract}

\section{\bf Introduction}
\vskip 3mm

 In 1971,  Pazy \cite{P1971} proved the following fixed point theorem for a  nonexpansive mapping  in Hilbert spaces, which is referred to as {\bf Pazy's Fixed Point Theorem}.

{\bf Theorem P.} {\it Let $H$ be a Hilbert space and $C$ be a nonempty closed convex subset of $H$. Let $T:C\rightarrow C$ be a nonexpansive mapping.
Then the following are equivalent:

\quad {\rm (i)}\, $\{T^nx\}$ is a  bounded sequence in $C$ for some $x\in C$.

\quad {\rm (ii)}\, $T$  has at least a fixed point in $C$.}
\vskip 3mm

Recently, many mathematical workers studied the Pazy's Fixed Point Theorem for many distinct types of nonlinear mappings.
In 2008, Kohsaka and Takahashi \cite{KT2008} showed   Pazy's Fixed Point Theorem of a  nonspreading mapping in a Hilbert space $H$. A mapping $T:C\rightarrow C$ is called a {\em nonspreading mapping} if
$$
2\|Tx-Ty\|^2\leq \|Tx-y\|^2+\|Ty-x\|^2\mbox{ for all $x,y\in C$.}
$$ They showed that $T$ has a fixed point whenever $\{T^nx\}$ is a  bounded sequence for some $x\in C$. The same result was obtained by  Takahashi \cite{WT2010} for a  hybrid mapping  in a Hilbert space $H$. A mapping $T:C\rightarrow C$ is called a {\em hybrid mapping} if
$$
\|Tx-Ty\|^2\leq \|x-y\|^2+\langle x-Tx,y-Ty\rangle\mbox{ for all $x,y\in C$.}
$$
In 2011,  Takahashi and Yao \cite{TY2011} proved Pazy's Fixed Point Theorem of a $TJ$-mapping in a Hilbert space $H$. A mapping $T:C\rightarrow C$ is called a {\em $TJ$-mapping} if
$$
2\|Tx-Ty\|^2  \leq  \|x-y\|^2 + \|Tx-y\|^2\mbox{ for all $x,y\in C$.}
$$
In 2012,  Lin and Wang \cite{LW2012} showed Pazy's Fixed Point Theorem  for an  $(a,b)$-monotone  mapping  in a Hilbert space $H$. A mapping $T:C\rightarrow C$ is called an {\em $(a,b)$-monotone  mapping} if for all $x, y \in C$,
$$
 \langle  x-y, Tx-Ty\rangle\geq   a \|Tx-Ty\|^2 + (1-a)\|x-y\|^2 -b \|x-Tx\|^2 -b \|y-Ty\|^2,
$$
 where  $a \in (\frac{1}{2}, \infty)$ and $b\in (- \infty,a) $ 
Very recently, Song et. al. \cite{SKC2016} firstly studied Pazy's Fixed Point Theorem  for a monotone nonexpansive mapping in a uniformly convex Banach space with the partial order ``$\leq$".  
Let $T$ be a mapping with domain $D(T)$ and range $R(T)$ in an ordered Banach space  $E$ endowed with  the partial order ``$\leq$". Then $T: D(T)\to R(T)$  is said to be 
{\em monotone nonexpansive} (\cite{BK15}) if for all $x,y\in D(T)$ with $x\leq y$,
$$Tx\leq Ty \mbox{ and }\|Tx-Ty\|\leq \|x-y\|.$$ Clearly, a monotone nonexpansive mapping may be discontinuous.
In 2015, Bachar and Khamsi \cite{BK15} introduced the concept of a monotone nonexpansive mapping and studied  common approximate fixed points of a monotone nonexpansive semigroup.  Dehaish and  Khamsi \cite{DK15} proved some weak convergence theorems of the Mann iteration for finding some order fixed points of  monotone nonexpansive mappings in  uniformly convex ordered Banach spaces.

The {\em Mann iteration} was introduced by Mann \cite{M53} in 1953 for finding a fixed point of a nonexpansive mapping $T$,
$$
x_{n+1}=\beta_{n}x_n+(1-\beta_{n})Tx_n\mbox{ for each positive integer $n$.}
$$
There had be many convergence conclusions of such an iteration in the past several decades. For more details, see Liu \cite{LL}, Narghirad et al. \cite{NWY}, Suzuki \cite{TS}, Song \cite{S1}, Opial \cite{ZO},  Kim et al. \cite{KLNC}, Okeke and Kim \cite{OK},  Berinde \cite{VB}, George and Shaini \cite{GS}, Gu and  Lu \cite{GL},  Zhou et al. \cite{ZZZ}, Song and Wang \cite{SW},  zhang and Su \cite{ZS},  Zhou \cite{Z08} and the reference therein.
Very recently, Song et. al. studied  weak convergence of  the Mann iteration  for monotone $\alpha$-nonexpansive mappings in a uniformly convex Banach space with the partial order ``$\leq$" under the coefficient condition
``$\limsup\limits_{n\to\infty}\beta_n(1-\beta_n)>0"  \mbox{ or }  ``\liminf\limits_{n\to\infty}\beta_n(1-\beta_n)>0",$ which excludes $\beta_n=0$ or $\frac1n$.\\

In this paper, we consider the Pazy's Fixed Point Theorem for monotone $\alpha$-nonexpansive mappings in a uniformly convex Banach space with the partial order ``$\leq$". That is, for a monotone $\alpha$-nonexpansive mapping $T$, the boundedness of the Picard iteration $\{T^nx_0\}$ for some $x_0$ with $x_0\leq Tx_0$ or $x_0\geq Tx_0$ implies that ordered fixed point set $F_{\geq}(T)\ne\emptyset$ or $F_{\leq}(T)\ne\emptyset$.
Furthermore,  for a monotone $\alpha$-nonexpansive mapping $T$ defined on a closed convex cone $P$ with respect to the partial order ``$\leq$", its fixed point set $F(T)\ne\emptyset$ if and only if  the Picard iteration $\{T^n0\}$ is bounded (where $0$ is the origin). Finally, under the frame of the partial order ``$\leq$" of a Banach space $(E,\|\cdot\|)$ defined by the normal cone $P$,  we will show the weak and strong  convergence of Picard iteration of a monotone $\alpha$-nonexpansive mapping.
\vskip 4mm

\section{\bf  Preliminaries and basic results}
\vskip 3mm

Throughout this paper, let $E$ be an ordered Banach space with the norm $\|\cdot\|$ and the partial order ``$\leq$".   Let $F(T)=\{x\in E: Tx=x\}$ denote the set of all fixed  points  of a mapping $T$.
\vskip 2mm

Let $E$ be a real Banach space and $P$ be a subset of $E$. $P$ is called a {\em closed convex cone} if $P$ is nonempty closed, and $P\ne \{0\}$ with  $P\cap(-P)=\{0\}$,
 and $ax+by \in P$ for all $x,y\in P$ and nonnegative real numbers $a,b$.
The {\em partial order} ``$\leq$", ``$<$" and ``$\ll$"  with respect to $P$ in $E$ are defined as follows: for all $x,y\in E$,
$$\aligned
x\leq y\,\,  &\mbox{ if and on if } \,\, y-x\in P,\\
x< y\,\,   &\mbox{ if and on if } \,\, y-x\in P \mbox{ and } y\neq x,\\
x\ll y \,\, &\mbox{ if and on if }  \,\, y-x\in  \mathring{P} \mbox{ whenever }\mathring{P}\ne\emptyset,
\endaligned$$
 where $ \mathring{P}$ is the interior of $P$. The  partial order ``$\geq$" is defined by $$x\geq y\,\,  \mbox{ if and on if } y\leq x.$$ Then we have \begin{equation}\label{eq:21}
 x=y\,\,  \mbox{ if and on if } y\leq x \mbox{ and }y\geq x.
 \end{equation}  
\vskip 2mm

In the sequel,  $[x,y]$ stands for  the order interval. An {\it order interval} $[x,y]$ for all $x,y\in E$ is given by
\begin{equation}\label{eq:22}
[x,y]=\{z\in E: x\leq z\leq y\}.
\end{equation}
Clearly, the order interval $[x,y]$  are closed and convex by the definition of the partial order  ``$\leq$",   which implies that
\begin{equation}\label{eq:23}
x\leq tx+(1-t)y\leq y
\end{equation}
for all $t$ with $0\leq t\leq 1$ and all $x,y\in E$  with $x\leq y$.
\vskip 2mm

\begin{Definition}(\cite{KD})\label{def:21} Let $(E,\|\cdot\|)$ be  an ordered Banach space with the partial order  ``$\leq$"  with respect to a cone $P$.
	\begin{itemize}
		\item[{(1)}] A sequence $\{x_n\}\subset E$ is called {\em monotonic}  if  $\{x_n\}$ is either increasing, i.e.,  $x_n\leq x_{n+1}$ for all positive integer $n$, or decreasing, i.e.,  $x_n\geq x_{n+1}$ for all positive integer $n$.
		\item[{(2)}] A subset $C$ of $E$ is called  {\em bounded above}  if $C$  has an upper bound with respect to ``$\leq$", i.e., there exists $y\in E$ such that $x\leq y$ for all $x\in C,$ and the least upper bound of $C$ with respect to ``$\leq$" is called  the supremum of $C$,  denote $\sup C$   if it exists.
\item[{(3)}] A subset $C$ of $E$ is called  {\em bounded below}  if $C$  has a lower bound with respect to ``$\leq$", i.e., there exists $z\in E$ such that $z\leq x$ for all $x\in C,$ and the greatest lower bound of $C$ with respect to ``$\leq$" is called  the infimum of $C$,  denote $\inf C$   if it exists.
		\item[{(4)}] The cone $P$ is called {\em normal} if  there exists a constant $\gamma>0$ such that
		$
		\|x\|\leq\gamma\|y\|$  for all $x,y\in E$  with $0\leq x\leq y
		$ or equivalently, if the order interval $[x,y]$ for all $x,y\in E$  with $x\leq y$ is bounded  with respect to the norm $\|\cdot\|$.
		\item[{(4)}]  The cone $P$ is called {\em minihedral}  if $\sup\{x,y\}$ exists for all $x,y\in E$, and {\em strongly minihedral}  if each set which is bounded above has a supremum.
		\end{itemize}
\end{Definition}
\vskip 2mm

\begin{lem}\label{lm:21}(\cite{JS})
	Let $(E,\|\cdot\|)$ be  an ordered Banach space with the partial order  ``$\leq$" with respect to a cone $P$.
\begin{itemize}
		\item[{(i)}] Assume that  $\{x_n\}$ and $\{y_n\}$ are two sequences on $E$ such that $$x_n\leq y_n\mbox{ for each positive integer } n.$$　If $x_n$ and $y_n$ weakly converges to $x$ and $y$, respectively,　then $$x\leq y.$$
\item[{(ii)}] The cone $P$ is normal if and only if there exists a equivalent norm $\|\cdot\|_1$ of $E$ such that
		$$
		\|x\|_1\leq \|y\|_1\mbox{  for all }x,y\in E\mbox{  with }0\leq x\leq y.$$
\end{itemize}
\end{lem}
\vskip 2mm

\begin{Definition}(\cite{BK15,SPKC2016})\label{def:22} Let $K$ be a nonempty closed convex subset of an ordered Banach space $(E,\leq)$. A mapping $T : K \to E$ is said to be:
	\begin{itemize}
\item[{(1)}] {\em monotone} (\cite{BK15}) if $Tx\leq Ty$ for all $x,y\in K$ with $x\leq y$;
\item[{(2)}] {\em monotone nonexpansive} (\cite{BK15}) if $T$ is monotone  and
$$\|Tx-Ty\|\leq \|x-y\|
$$
for all $x,y\in K$ with $x\leq y$.
\item[{(3)}]{\it monotone $\alpha$-nonexpansive}(\cite{SPKC2016}) if  $T$ is monotone and for some $\alpha<1$,
	$$
	\|Tx-Ty\|^2\leq\alpha\|Tx -y\|^2 +\alpha\|Ty-x\|^2 + (1 -2\alpha) \|x -y\|^2
	$$
	for all $x, y \in K$ with $x\leq y$;
\item[{(4)}] {\it monotone quasi-nonexpansive}(\cite{SPKC2016}) if $F(T)\ne\emptyset$ and
	$$
	\|Tx-p\|\leq\|x-p\|
	$$
	for all $p\in F(T)$ and all $x\in K$ with $x\leq p$ or $x\geq p$.\end{itemize}
\end{Definition}
\vskip 2mm

Obviously, a monotone nonexpansive mapping is  monotone $0$-nonexpansive. In the sequel, we will use the fixed point sets with the partial order $F_{\leq}(T)$ and $F_{\geq}(T)$ given by
$$
F_{\leq}(T)=\{p\in F(T):p\leq x \mbox{ for some }x\in K\}$$  and   $$F_{\geq}(T)=\{p\in F(T):x\leq p\mbox{ for some }x\in K\}.
$$
\vskip 2mm

\begin{lem}\label{lm:22}(\cite[Lemma 2.1]{SPKC2016})
	Let $K$ be a nonempty closed convex subset of an ordered Banach space $(E,\leq )$ and  $T : K\to K$ be a monotone
	$\alpha$-nonexpansive mapping. Then
	\begin{itemize}\item[(1)] $T$ is monotone quasi-nonexpansive if $F_{\leq}(T)\ne\emptyset$ or $F_{\geq}(T)\ne\emptyset$;
		\item[(2)] for all $x,y\in K$  with $x\leq y$ (or $y\leq x$),
		$$
		\aligned
		&\quad \|Tx-Ty\|^2\\
		&\leq \| x-y\|^2+\frac{2\alpha}{1-\alpha}\|Tx-x\|^2+\frac{2|\alpha|}{1-\alpha}\|Tx-x\|(\|x-y\|+\|Tx-Ty\|).
		\endaligned
		$$
	\end{itemize}
\end{lem}
\vskip 2mm

\begin{Definition}\label{def:23} Let $E$ ba a Banach space. Then \begin{itemize}\item[(1)] a function $\delta_E:[0,2]\to[0,1]$ is said to be
the {\em modulus of convexity of $E$} if
$$\delta_E(\varepsilon)=\inf\{1-\frac{\|x+y\|}2;\ \|x\|\leq1,\
\|y\|\leq1,\ \|x-y\|\geq\varepsilon\};$$
\item[(2)] a number $\varepsilon_0(E)$ is said to be the {\em characteristic of convexity of $E$} if $$\varepsilon_0(E)=\sup\{\varepsilon\in[0,2];\delta_E(\varepsilon)=0\}.$$
\end{itemize}
\end{Definition}

\begin{Definition}\label{def:24}
	A Banach space $E$ is said to be
\begin{itemize}\item[(1)] {\em uniformly convex} if $\delta_E(\varepsilon)>0$
 for all $\varepsilon\in(0,2]$ or equivalently if $\varepsilon_0(E)=0$;
	\item[(2)]{\em strictly convex} if $\delta_E(2)=1$.
\end{itemize}
\end{Definition}
\vskip 2mm

The following properties of the modulus of convexity of a Banach space $E$ were found in the references \cite{ARS,GK,KS}.
\vskip 2mm

\begin{lem} \label{lm:23} Let $E$ be a Banach space with the modulus of convexity $\delta_E(\cdot)$ and the characteristic of convexity $\varepsilon_0$.  Then we have the following:
 \begin{itemize}\item[(i)] $\delta_E(\cdot)$ is continuous on $[0,2)$;
 \item[(ii)] $\delta_E(\cdot)$ is strictly increasing on $[\varepsilon_0,2]$;
\item[(iii)]If, in addition, $E$ is uniformly convex, $r>0$ and $x,y\in E$ with $ \|x\|\leq r,\ \|y\|\leq r$, then \begin{equation}\label{eq:24}
	\|\lambda x+(1-\lambda)y\|\leq
r\left[1-2\min\{\lambda,1-\lambda\}\delta_E\left(\frac{\|x-y\|}r\right)\right]\mbox{ for all }\lambda\in[0,1].
	\end{equation}
	In particular, taking $\lambda=\frac12$,
	\begin{equation}\label{eq:25}
	\Big\|\frac{x+y}2\Big\|\leq r\left[1-\delta_E\left(\frac{\|x-y\|}r\right)\right].
	\end{equation}	
\item[(iv)] Each  uniformly convex Banach space has the Kadec-Klee property, i.e.,
$$\mbox{\em weak-}\lim_{n\to\infty}x_n=x\mbox{  and }\lim_{n\to\infty}\|x_n\|=\|x\|\mbox{ implies }\lim_{n\to\infty}x_n=x.$$
\end{itemize}
\end{lem}
\vskip 2mm

The following conclusion is well known. For the more detail, see  the references \cite{KD,KS,WT}.
\vskip 2mm

\begin{lem}\label{lm:24}   Let $C$ be a nonempty closed convex subset of a reflexive Banach space $E$. Assume that $f:C\to (-\infty,+\infty)$ is a proper convex lower semi-continuous  and coercive function $($i.e., $\lim_{\|x\|\to \infty}f(x)=\infty$$)$. Then there exists $x\in C$  such that
	$$
	f(x)=\inf_{y\in C} f(y).
	$$
\end{lem}
\vskip 4mm

\section{\bf Pazy type Fixed Point Theorems with the partial order}
\vskip 3mm

Let $K$ be a nonempty closed convex subset of an ordered Banach space $(E,\leq)$ and let $T:K\to K$ be a  monotone $\alpha$-nonexpansive mapping. In this section, we
consider the Picard iteration defined by
\begin{equation}\label{eq:31}
x_{n+1}=Tx_n=T^nx_0
\end{equation}
for fixed $x_0\in K$ and each positive integer $n\geq 1$.
\begin{Definition}\label{def:31} Let $T$ be a mapping on a Banach space $E$ and $x_0\in D(T)$, the domain of $T$. Then the set of points, denoted by $O(x_0)$ is called the {\em orbit of $x_0$ under $T$} if   $$O(x_0)=\{T^nx_0;n=0,1,2,\cdots\},$$ and its closure is called the {\em closed orbit}.
\end{Definition}

\vskip 2mm
The following technical lemma plays a key role in the proof of main results of this work, which may be found in the reference \cite{JS}. For the completeness, we give its proof. \vskip 2mm

\begin{lem} \label{lm:31} Let $(E,\|\cdot\|)$ be a Banach space with the partial order  ``$\leq$". Assume that a sequence $\{x_n\}\subset E$ is  monotonic.  \begin{itemize}\item[(i)] If there exists a subsequence $\{x_{n_k}\}\subset\{x_n\}$ such that $x_{n_k}$ weakly converges to some point $x\in E$, then $x_n\leq x$ for all positive integer $n$ when $\{x_n\}$ is increasing or $x_n\geq x$ for all positive integer $n$ when $\{x_n\}$ is decreasing.
\item[(ii)] If $(E,\|\cdot\|)$ is reflexive and $\{x_n\}$ is bounded  with respect to the norm $\|\cdot\|$, then $x_n$ weakly converges to some point $x\in E$.
\end{itemize}
\end{lem}

\begin{proof}
(i) Assume that  $\{x_n\}$ is increasing, i.e.,  $x_n\leq x_{n+1}$ for all positive integer $n$.  Then for given $n$, there exists a positive integer $k_0$ such that $n<n_{k_0}$, and hence, $x_n\leq x_{n_{k_0}}\leq x_{n_k}$ for all $k>k_0$. It follows from Lemma \ref{lm:21} (i) and the weak convergence of $x_{n_k}$  that $x_n\leq x.$ Since $n$ is arbitrary, $x_n\leq x$ for all positive integer $n$.

 Using the same proof technique, it is easy to prove that $x_n\geq x$ for all positive integer $n$ when $\{x_n\}$ is decreasing.

(ii) Without loss of generality, we may assume that  $\{x_n\}$ is increasing. It follows from the reflexivity of $E$ and the boundedness of $\{x_n\}$ that $\{x_n\}$ is weakly sequentially compact. Then we may choose two  subsequences $\{x_{n_k}\}$ and $\{x_{n_j}\}$ in $\{x_n\}$ such that $x_{n_k}$ and $\{x_{n_j}\}$ weakly converge to  $x$ and $y$, respectively.  We must have $x=y$. In fact, using the the same proof technique as (i), for fixed $n_k$, there is large enough $n_j$ such that $x_{n_{k}}\leq x_{n_j}$. Then by Lemma \ref{lm:21} (i), we have $x_{n_k}\leq y$, and so for all $k$, $x_{n_k}\leq y$. Again using Lemma \ref{lm:21} (i), we have $x\leq y$. Using the same proof technique, we also have $y\leq x$. Thus $x=y.$ Which means that any subsequences $\{x_{n_k}\}$ of $\{x_n\}$ weakly converges to  $x$, and hence, $\{x_n\}$ weakly converges to  $x$.
\end{proof}\vskip 2mm

Now we prove some existence theorems of fixed points with the partial order  ``$\leq$" of  monotone $\alpha$-nonexpansive mappings in a  partial order Banach space.
\vskip 2mm

\begin{thm} \label{t:32} Let $K$ be a nonempty  closed convex subset of a uniformly convex  Banach space $(E,\|\cdot\|)$ with the partial order ``$\leq$" and  $T : K\to K$ be a
	monotone $\alpha$-nonexpansive mapping.  \begin{itemize}\item[(i)] If $x_0\leq Tx_0$ and the orbit $O(x_0)$ is bounded with respect to the norm $\|\cdot\|$,
	then $F_{\geq}(T)\ne\emptyset$.
\item[(ii)] If $F_{\geq}(T)\ne\emptyset$,
	then there exists $x_0\in K$ such that the orbit $O(x_0)$ is bounded with respect to the norm $\|\cdot\|$.
\end{itemize}
\end{thm}
\vskip 2mm

\begin{proof} (i)  Since $x_0\leq Tx_0$, it follows from the monotonicity of $T$ that  $Tx_0\leq T^2x_0,$ and so $$T^2x_0\leq T^3x_0,\, \cdots, \, T^nx_0\leq T^{n+1}x_0,\, \cdots.$$ Namely,
$$x_0\leq Tx_0\leq T^2x_0\leq T^3x_0 \leq \cdots  \leq T^nx_0\leq T^{n+1}x_0  \cdots.$$
By Lemma \ref{lm:31}, there exists $x\in E$ such that $T^nx_0$ weakly converges to $x$ and $T^nx_0\leq x$ for all positive integer $n$. Since the closed convexity of  $K$ implies that $K$ is weakly sequentially closed,  we have $x\in K.$

	Let $K_n=\{y\in K; T^nx_0\leq y\}$ for all positive integer $n$. Clearly, for each $n$, $K_n$ is closed convex and $x\in K_n$ for all positive integer $n$. Let $C=\bigcap\limits_{n=1}^\infty K_n$. Clearly, $C\subset K$ is a nonempty closed convex  ($x\in C$). Since the orbit $O(x_0)=\{T^nx_0; n=0,1,2,\cdots\}$ is bounded with respect to the norm $\|\cdot\|$, we may define a function $f:C\to[0,+\infty)$ as follows
	$$
	f(y)=\limsup_{n\to\infty}\|T^nx_0-y\|\mbox{ for all $y\in C.$}
	$$
It is obvious that $f$ is a proper convex,  continuous  and coercive function.
	  It follows from Lemma \ref{lm:24} that there exists $z\in C$ such that
	\begin{equation}\label{eq:32}
	f(z)=\limsup_{n\to\infty}\|T^nx_0-z\|=\inf_{y\in C} f(y)=r.
	\end{equation}
	By the definition of $C$, we have
	$$
	T^nx_0\leq z\mbox{ for all positive integer }n,
	$$
	and hence, $$
	T^nx_0\leq T^{n+1}x_0\leq Tz\mbox{ for all positive integer }n,
	$$
which implies that $Tz\in C$. Thus $\dfrac{z+Tz}2\in C$, and so, it follows from  \eqref{eq:32} that
	\begin{equation}\label{eq:33}
	r=f(z)\leq f\Big(\frac{z+Tz}2\Big)\quad \mbox{and}\quad r=f(z)\leq f(Tz).
	\end{equation}
	
On the other hand, it follows from the definition of $T$ that
	$$
	\|T^{n+1}x_0-Tz\|^2\leq\alpha\|T^{n+1}x_0 -z\|^2 +\alpha\|Tz-T^nx_0\|^2 + (1 -2\alpha) \|T^nx_0 -z\|^2,
	$$
and so,
\begin{align}
	(\limsup_{n\to\infty}\|T^{n+1}x_0-Tz\|)^2\leq&\alpha(\limsup_{n\to\infty}\|T^{n+1}x_0 -z\|)^2 +\alpha(\limsup_{n\to\infty}\|Tz-T^nx_0\|)^2\nonumber\\
& + (1 -2\alpha) (\limsup_{n\to\infty}\|T^nx_0 -z\|)^2,\nonumber
	\end{align}
that is, $$(1 -\alpha)(\limsup_{n\to\infty}\|Tz-T^nx_0\|)^2\leq(1 -\alpha) (\limsup_{n\to\infty}\|T^nx_0 -z\|)^2.$$
Thus, we have
	\begin{equation}\label{eq:34}
	f(Tz)=\limsup_{n\to\infty}\|T^nx_0-Tz\|\leq \limsup_{n\to\infty}\|T^nx_0 -z\|=f(z)=r.
	\end{equation}
Combining \eqref{eq:33}	and \eqref{eq:34}, we obtain
$$r=f(z)=f(Tz)\geq0.$$

We claim $z=Tz.$ In fact, suppose $r=0$. Then $$\lim_{n\to\infty}\|T^nx_0-Tz\|= \lim_{n\to\infty}\|T^nx_0 -z\|=0,$$
which means $z=Tz.$ Next  suppose $r>0.$
	It follows from the definition of the upper limit ``$\limsup$" that for all $\varepsilon>0$, there exists a positive integer $N$ such that
	$$\|T^nx_0-Tz\|<r+\varepsilon\mbox{ and }\|T^nx_0-z\|<r+\varepsilon \mbox{ for all positive integer }n>N.$$
It follows from Lemma \ref{lm:23} that for all positive integer $n>N$, $$\Big\|T^nx_0-\frac{z+Tz}2\Big\|=\Big\|\frac{(T^nx_0-z)+(T^nx_0-Tz)}2\Big\|\leq(r+\varepsilon)\left[1-\delta_E\left(\frac{\|z-Tz\|}{r+\varepsilon}\right)\right].$$
Without loss of generality, we may assume that  $\varepsilon<1$. Then the above inequality can be rewritten  as follow:
$$\Big\|T^nx_0-\frac{z+Tz}2\Big\|\leq(r+\varepsilon)\left[1-\delta_E\left(\frac{\|z-Tz\|}{r+1}\right)\right],$$
and so, we have
 $$f\Big(\frac{z+Tz}2\Big)=\limsup_{n\to\infty}\Big\|T^nx_0-\frac{z+Tz}2\Big\|\leq(r+\varepsilon)\left[1-\delta_E\left(\frac{\|z-Tz\|}{r+1}\right)\right].$$
From \eqref{eq:33}, it follows that
	$$
	r\delta_E\left(\frac{\|z-Tz\|}{r+1}\right)\leq(r+\varepsilon)\delta_E\left(\frac{\|z-Tz\|}{r+1}\right)\leq r-f\Big(\frac{z+Tz}2\Big)+\varepsilon
	\leq\varepsilon.$$
Since $\varepsilon$ is arbitrary, we have $$r\delta_E\left(\frac{\|z-Tz\|}{r+1}\right)=0,$$	and hence,  $z=Tz$ by  Lemma \ref{lm:23} (ii). Clearly, $z\geq x_0$, and so $z\in F_{\geq}(T).$

(ii) Let $z\in F_{\geq}(T)$. Then there exists a $x_0\in K$ such that $x_0\leq z$. It follows from the monotonicity of $T$ that $Tx_0\leq Tz=z$, and hence, $T^2x_0=T(Tx_0)\leq Tz=z$. By such analogy, we have
$$T^nx_0\leq Tz=z\mbox{ for all positive integer }n.$$
From Lemma \ref{lm:22}, it follows that $$\|T^nx_0-z\|\leq \|x_0-z\|\mbox{ for all positive integer }n,$$
which implies the boundedness of the sequence $\{T^nx_0\}$. This completes the proof.
\end{proof}
\vskip 2mm

Using the same proof technique of Theorem \ref{t:32} (the only change is $K_n=\{y\in K:y\leq  T^nx_0\}$ and $T^{n+1}x_0\leq T^nx_0$),  the following theorem is easy to be proved.
\vskip 2mm

\begin{thm} \label{t:33} Let $K$ be a nonempty  closed convex subset of a uniformly convex  Banach space $(E,\|\cdot\|)$ with the partial order ``$\leq$" and  let $T : K\to K$ be a
	monotone $\alpha$-nonexpansive mapping. \begin{itemize}\item[(i)] If $x_0\geq Tx_0$ and the orbit $O(x_0)$ is bounded with respect to the norm $\|\cdot\|$,
	then $F_{\leq}(T)\ne\emptyset$.
\item[(ii)] If $F_{\leq}(T)\ne\emptyset$,
	then there exists $x_0\in K$ such that the orbit $O(x_0)$ is bounded with respect to the norm $\|\cdot\|$.
\end{itemize}
\end{thm}
\vskip 2mm

\begin{thm} \label{t:34} Let $(E,\|\cdot\|)$ be a uniformly convex ordered Banach space with the partial order ``$\leq$" with respect to closed convex cone $P$ and let
	$T : P\to P$ be a monotone $\alpha$-nonexpansive mapping.
	Then $F(T)\ne\emptyset$ if and only if  the orbit $O(0)$ is bounded with respect to the norm $\|\cdot\|$.
\end{thm}
\vskip 2mm

\begin{proof} Since $T(P)\subset P$, for all $x\in F(T)$, we have $x\geq0$, and so $x\in F_{\geq}(T)$. Then $F(T)\subset F_{\geq}(T)\subset F(T)$,
and so,	$F(T)=F_{\geq}(T)$.
	
	 ``Sufficiency". It follows from the definition of the partial order ``$\leq$" that $x_0=0\leq T0=Tx_0$. Then the conclusion directly follows from Theorem \ref{t:32}.
	
	 ``Necessity". Take $x\in F(T)=F_{\geq}(T)$. Then $x\geq0$, and hence, $$x=Tx\geq T0, x=Tx\geq T^20, \cdots, x=Tx\geq T^n0, \cdots.$$
	 That is, $x\geq T^n0$ for all positive integer $n$.
	  It follows from Lemma \ref{lm:22} that $$\|T^{n+1}0-x\|\leq \|T^n0-x\|\leq\cdots\leq\|0-x\|=\|x\|.$$
	This means that the orbit $O(0)$ is bounded with respect to the norm $\|\cdot\|$.
\end{proof}
\vskip 4mm

\section{\bf Weak and strong  convergence of Picard iteration}
\vskip 3mm

\begin{Definition}\label{def:41} Let  $(E,\|\cdot\|)$ be a Banach space with the partial order ``$\leq$".  The norm $\|\cdot\|$ is called  {\em monotonic} if
	$$
	\|x\|\leq \|y\|\mbox{  for all }x,y\in E\mbox{  with }0\leq x\leq y.$$
\end{Definition}
 It is well known that if the cone $P$ is normal, by Lemma \ref{lm:21},  there exists a equivalent norm $\|\cdot\|_1$ of $E$ such that
 $$
 \|x\|_1\leq \|y\|_1\mbox{  for all }x,y\in E\mbox{  with }0\leq x\leq y.$$
\vskip 3mm

\begin{thm} \label{t:41} Let $K$ be a nonempty  closed convex subset of a uniformly convex  Banach space $(E,\|\cdot\|)$ with the partial order ``$\leq$" with respect to the normal cone $P$ and let	$T : K\to K$ be a monotone $\alpha$-nonexpansive mapping. Assume that  $x_0\leq Tx_0$ and the orbit $O(x_0)$ is bounded with respect to the norm $\|\cdot\|$. If the norm $\|\cdot\|$ is  monotonic, then $T^nx_0$ weakly converges to some point $z\in F_{\geq}(T)$.
\end{thm}
\vskip 2mm

\begin{proof}
Using the same proof technique of Theorem \ref{t:32}, we obtain the following:
	\begin{itemize}
		\item[(i)]  $T^nx_0$ weakly converges to some point  $x\in  K$ and $T^nx_0\leq x$ for all positive integer $n$. Furthermore,
 $$x\in C=\bigcap\limits_{n=1}^\infty\{y\in K; T^nx_0\leq y\}.$$
	 \item[(ii)] Let $f(y)=\limsup\limits_{n\to\infty}\|T^nx_0-y\|$  for all $y\in C.$ Then there exists $z\in C$ such that $z=Tz\geq T^nx_0$ for all positive integer $n$ and $$f(z)=\inf\{f(y);y\in C\}\leq f(x).$$
	\end{itemize}

Next we prove $z=x$. It follows from Lemma \ref{lm:21} (i) that $$ T^nx_0\leq x\leq z=Tz.$$
Thus, we have 	$$0\leq x-T^nx_0\leq  z-T^nx_0, $$ which implies that $$\| x-T^nx_0\|\leq \| z-T^nx_0\| $$ by the monotonicity of the norm $\|\cdot\|$.
Consequently, $$f(x)=\limsup_{n\to\infty}\| x-T^nx_0\|\leq \limsup_{n\to\infty}\| z-T^nx_0\|=f(z)\leq f(x), $$
and hence, $f(x)=f(z)=r$.  Using the same proof technique of Theorem \ref{t:32}, we have $x=z$ also. This completes the proof.
\end{proof}

Similarly, we also the following theorem.

\begin{thm} \label{t:42} Let $K$ be a nonempty  closed convex subset of a uniformly convex  Banach space $(E,\|\cdot\|)$ with the partial order ``$\leq$" and let	$T : K\to K$ be a monotone $\alpha$-nonexpansive mapping. Assume that  $x_0\geq Tx_0$ and the orbit $O(x_0)$ is bounded with respect to the norm $\|\cdot\|$.  If the norm $\|\cdot\|$ is  monotonic, then $T^nx_0$ weakly converges to some point $z\in F_{\geq}(T)$.
\end{thm}
\vskip 2mm

Let the intersection of the domain $D(T)$ of monotone $\alpha$-nonexpansive mapping $T$ and the cone $P$ is nonempty,  i.e., $D(T)\cap P\ne\emptyset$. The following strongly convergent theorems may be obtained.

\begin{thm} \label{t:43} Let $K$ be a nonempty  closed convex subset of a uniformly convex  Banach space $(E,\|\cdot\|)$ with the partial order ``$\leq$"  and let	$T : K\to K$ be a monotone  $\alpha$-nonexpansive mapping. Assume that  $0\leq x_0\leq Tx_0$ and the orbit $O(x_0)$ is bounded with respect to the norm $\|\cdot\|$.  If the norm $\|\cdot\|$ is  monotonic, then $T^nx_0$ strongly converges to some point $z\in F_{\geq}(T)$.
\end{thm}
\vskip 2mm
\begin{proof} It follows from Theorem \ref{t:41} that there exists $z\in F_{\geq}(T)$ such that $T^nx_0$ weakly converges to some point  $z\in  K$ and
	\begin{equation}\label{eq:43}0\leq x_0\leq T^nx_0\leq T^{n+1}x_0\leq z\mbox{ for all positive integer }n.\end{equation}
	Then  by the monotonicity of the norm, we have
\begin{equation}\label{eq:44}0\leq \|x_0\|\leq \|T^nx_0\|\leq \|T^{n+1}x_0\|\leq \|z\|\mbox{ for all positive integer }n.\end{equation}
That is, the sequence $\{\|T^nx_0\|_1\}$ is monotone increasing and bounbed. So there is a real number $\gamma$ such that \begin{equation}\label{eq:45}\lim_{n\to \infty}\|T^nx_0\|=\gamma\leq\|z\|.\end{equation}
From the weakly lower semi-continuity  of the norm, it follows that
	$$\|z\|_1\leq\liminf_{n\to \infty}\|T^nx_0\| =\lim_{n\to \infty}\|T^nx_0\|=\gamma\leq\|z\|,$$
	and hence, $$\lim_{n\to \infty}\|T^nx_0\|=\|z\|.$$
By Lemma \ref{lm:23} (iv), we have $\lim\limits_{n\to \infty}T^nx_0=z.$ This completes the proof.
\end{proof}
\vskip 2mm

\begin{thm} \label{t:44} Let $K$ be a nonempty  closed convex subset of a uniformly convex  Banach space $(E,\|\cdot\|)$ with the partial order ``$\leq$"  and let	$T : K\to K$ be a monotone  $\alpha$-nonexpansive mapping. Assume that  $0\geq x_0\geq Tx_0$ and the orbit $O(x_0)$ is bounded with respect to the norm $\|\cdot\|$.  If the norm $\|\cdot\|$ is  monotonic, then $T^nx_0$ strongly converges to some point $z\in F_{\leq}(T)$.
\end{thm}
\vskip 2mm
\begin{proof} It follows from Theorem \ref{t:42} that there exists $z\in F_{\leq}(T)$ such that $T^nx_0$ weakly converges to  $z$ and
	$$0\geq x_0\geq T^nx_0\geq T^{n+1}x_0\geq z\mbox{ for all positive integer }n.$$
	Then  $$0\leq -x_0\leq -T^nx_0\leq -T^{n+1}x_0\leq -z\mbox{ for all positive integer }n,$$
	and so, $$0\leq \|x_0\|\leq \|T^nx_0\|\leq \|T^{n+1}x_0\|\leq \|z\|\mbox{ for all positive integer }n.$$
The remainder of the proof are the same as ones of Theorem \ref{t:43}, we omit it.
\end{proof}
\vskip 2mm

Combining Theroem \ref{t:34} and \ref{t:41}, the following conclusion is easy to be showed.

\vskip 2mm
\begin{cor} \label{co:45} Let $(E,\|\cdot\|)$ be a uniformly convex  Banach space  with the partial order ``$\leq$" with respect to the closed convex cone $P$ and let	$T : P\to P$ be a monotone  $\alpha$-nonexpansive mapping with $F(T)\ne\emptyset$.   If the norm $\|\cdot\|$ is  monotonic, then $T^n0$ strongly converges to some point $z\in F(T)$.
\end{cor}

\vskip 2mm
\begin{cor} \label{co:46} Let $(E,\|\cdot\|)$ be a uniformly convex  Banach space  with the partial order ``$\leq$" with respect to the closed convex $P$ and let	$T : P\to P$ be a monotone nonexpansive mapping with $F(T)\ne\emptyset$.  If the norm $\|\cdot\|$ is  monotonic, then for all $x\in P$ with  $x\leq Tx$, $T^nx$ strongly converges to some point $z\in F(T)$.
\end{cor}
\vskip 2mm
\begin{proof} It follows from Theroem \ref{t:34} that the orbit $O(0)$ is bounded with respect to the norm $\|\cdot\|$. Then by the monotonicity of $T$, for all $x\in P$ (i.e.,$0\leq x$), we have
	$T0\leq Tx$, $T^20\leq T^2x,$ and so by the same manner, we obtain
	$$T^n0\leq T^nx\mbox{ for all positive integer }n.$$
Since  	$T$ be a monotone nonexpansive mapping, we have
$$\|T^n0-T^nx\|\leq\|T^{n-1}x-T^{n-1}0\|\leq\cdots\leq\|x-0\|=\|x\|\mbox{ for all positive integer }n.$$
So the orbit $O(x)$ is bounded with respect to the norm $\|\cdot\|$. Then the conclusion directly follows from Theorem \ref{t:43}.
\end{proof}
\vskip 2mm


\end{document}